\documentclass[a4paper,11pt]{article}%{article}%

\usepackage[top=3.0cm, bottom=3.0cm, inner=3.0cm, outer=3.0cm, includefoot]{geometry}

\usepackage[latin1]{inputenc}
\usepackage[T1]{fontenc}
\usepackage{verbatim}

\usepackage{geometry}
\usepackage{amssymb}
\usepackage{amsmath}
\usepackage{graphicx}
\usepackage{amsthm}

\usepackage{color}

\usepackage{enumerate}

\newcommand{\IR}{{\mathbb{R}}}

\newcommand{\IN}{{\mathbb{N}}}

\newcommand{\posR}{\mathbb{R}^+}
\newcommand{\nnegR}{\mathbb{R}^+_0}
\newcommand{\sol}{u \in C^1( V \times \nnegR)}

\newcommand{\eChar}{\begin{enumerate}[(i)]}
\newcommand{\eCharR}{\begin{enumerate}[(a)]}
\newcommand{\eBr}{\begin{enumerate}[(1)]}

\newcommand{\zi}{z_{i'}}%{\widetilde{z_i}}%{Z_i}%
\newcommand{\LL}{\mathcal L}
\newcommand{\II}{\posR}

\newcommand{\Abstract}
{
We show a connection between the $CDE'$ inequality introduced in \cite{Horn2014} and the $CD\psi$ inequality established in \cite{Münch2014}. In particular, we introduce a $CD_\psi^\varphi$ inequality as a slight generalization of $CD\psi$ which turns out to be equivalent to $CDE'$ with appropriate choices of $\varphi$ and $\psi$.
We use this to prove that the $CDE'$ inequality implies the classical $CD$ inequality on graphs, and that the $CDE'$ inequality with curvature bound zero holds on Ricci-flat graphs.
}

\title
{
Remarks on curvature dimension conditions on graphs
}

\author{Florentin Münch}
\date{\today}

\theoremstyle{plain}
\newtheorem{lemma}{Lemma}[section]
\newtheorem{theorem}[lemma]{Theorem}
\newtheorem{proposition}[lemma]{Proposition}
\newtheorem{corollary}[lemma]{Corollary}

\theoremstyle{definition}

\newtheorem{example}[lemma]{Example}

\newtheorem{defn}[lemma]{Definition}

\numberwithin{equation}{section}

\begin{document}

\maketitle
\begin{abstract}
  \Abstract
\end{abstract}
%\tableofcontents

%\section*{Preface}

%\markboth{PREFACE}{PREFACE} 
%I want to thank my family, my advisors Prof. Daniel Lenz and Dr. Matthias Keller and all other people who supported me while writing my Master article.

%\part{Li-Yau-Harnack inequalities on semigroups}

\pagestyle{plain}

%\addcontentsline{toc}{section}{Abstract}

%--------------------------------------------------------------------------------------------------------------------------------------------------------------
%--------------------------------------------------------------------------------------------------------------------------------------------------------------
\section{Introduction}

There is an immense interest in the heat equation on graphs. In
this context, curvature-dimension conditions have attracted
particular attention. In particular, recent works \cite{Bauer2013,Horn2014,Münch2014} have
introduced a variety of such conditions. In this note, we will extend ideas of \cite{Münch2014} to show
a connection between them (Proposition~\ref{PCC} and Section~\ref{SSCC}). Moreover, we
will prove that Ricci-flat graphs satisfy the $CDE'$ condition
(Section~\ref{SRF}). 

Throughout the note,  we will use  notation and definitions
introduced in \cite{Bakry1985,Bauer2013, Chung1996,Horn2014,Münch2014} which can
 also be found in the appendix.

{\bf Acknowledgements}

I wish to thank Matthias Keller and Daniel Lenz for their support and for sharing ideas in many fruitful discussions.
%, who have been the supervisors of my master thesis which is presented here in this article, for many useful discussions and for providing a very enjoyable and constructive atmosphere. Moreover, I %wish to acknowledge Matthias Keller for proposing the topic of my master thesis. 
%--------------------------------------------------------------------------------------------------------------------------------------------------------------
%--------------------------------------------------------------------------------------------------------------------------------------------------------------
\section{The connection between the $CDE'$ and the $CD\psi$ inequality}

First, we consider the connection between 
$\Gamma$ (cf. Definition~\ref{DG}) and $\Gamma^\psi$ (cf. Definition~\ref{DGpsi}),
 and between
$\widetilde{\Gamma}_2$ (cf. Definition~\ref{DGT}) and $\Gamma_2^\psi$ (cf. Definition~\ref{DefG2}).

%\begin{defn}
%The operator $ \widetilde{\Gamma}_2 : C^+(V) \to C(V)$ is defined by
%\[
%\widetilde{\Gamma}_2 (f) := \Gamma_2(f) - \Gamma\left(f, \frac {\Gamma(f)} f \right).
%\]
%\end{defn}

\begin{lemma} \label{LTG}
For all $f \in C^+(V)$,
\begin{eqnarray}
f \Gamma^{\sqrt{\cdot}} (f) &=& \Gamma(\sqrt f), 						\label{BauerG}			\\	
f \Gamma_2^{\sqrt{\cdot}} (f) &=& \widetilde{\Gamma}_2{(\sqrt f)}.  \label{BauerG2}
\end{eqnarray}
\end{lemma}

\begin{proof}
Let $f \in C^+(V)$ and $x \in V$. Then for the proof of (\ref{BauerG}),
\[
2\left[f \Gamma^{\sqrt{\cdot}} (f)\right](x) = 2f(x) \left[\frac {\Delta f}{2f} - \Delta \sqrt{\frac f {f(x)}}\right](x) = \left[\Delta f - 2\sqrt f \Delta \sqrt f \right](x) = 2 \Gamma(\sqrt f)(x).
\]

Next, we prove (\ref{BauerG2}).
In \cite[(4.7)]{Bauer2013}, it is shown that for all positive solutions $\sol$ to the heat equation, one has
\[
2 \widetilde{\Gamma}_2{\sqrt u} = \mathcal L (\Gamma \sqrt u).
\]
Now, we set $u:=P_t f$ and we apply the above proven identity (\ref{BauerG}) and the identity $2u\Gamma_2^\psi (u) = \mathcal L(u\Gamma^\psi (u))$ (cf. \cite[Subsection 3.2]{Münch2014}) to obtain
\begin{eqnarray*}
2 \widetilde{\Gamma}_2{(\sqrt f)} = \left[\mathcal L (\Gamma (\sqrt u))\right]_{t=0} = \left[\mathcal L (u \Gamma^{\sqrt{\cdot}} (u))\right]_{t=0} = 2 f \Gamma_2^{\sqrt{\cdot}} (f).
\end{eqnarray*}
This finishes the proof.
\end{proof}

The following definition extends the $CD\psi$ inequality to compare it to the $CDE'$ inequality.

\begin{defn}[$CD^\varphi_\psi$ condition]
Let $d \in (0,\infty]$ and $K\in \IR$. Let $\varphi,\psi \in C^1(\IR^+)$ be concave functions. A graph $G=(V,E)$ satisfies the $CD^\varphi_\psi (d,K)$ condition, if for all $f \in C^+(V)$,
\[
\Gamma_2^\psi(f) \geq \frac 1 d \left(\Delta^\varphi f \right)^2 + K \Gamma^\psi(f).
\]
\end{defn}

Indeed, this definition is an extension of $CD\psi$ which is equivalent to $CD_\psi^\psi$.

\begin{proposition} \label{PCC}
Let $G=(V,E)$ be a graph, let $d \in (0,\infty]$ and $K\in \IR$. Then, the following statements are equivalent.
\eChar
	\item   \label{1}
	  $G$ satisfies the $CDE'(d,K)$ inequality. 
	\item   \label{2}
	  $G$ satisfies the $CD^{\log }_{\sqrt{\cdot}}(4d,K)$ inequality. 
\end{enumerate}
\end{proposition}

\begin{proof}
By definition, the $CDE'(d,K)$ inequality is equivalent to
\[
\widetilde\Gamma_2(f)		\geq \frac 1 d 	f^2 \left(\Delta \log f\right)^2 + 		K\Gamma(f)                              , \quad f \in C^+(V).
\] 
By replacing $f$ by $\sqrt f$ (all allowed $f \in C(V)$ are strictly positive), this is equivalent to 
\[
\widetilde\Gamma_2(\sqrt{f})			 \geq \frac 1 d 	f \left(\Delta \log \sqrt{f}\right)^2 + K\Gamma(\sqrt f)					, \quad f \in C^+(V).
\]
By applying Lemma~\ref{LTG} and the fact that $\Delta^{\log} = \Delta \circ \log$, this is equivalent to
\[
f\Gamma_2^{\sqrt{\cdot}}(f)			 \geq \frac 1 {4d} f \left(\Delta^{\log} f \right)^2 + f K\Gamma^{\sqrt{\cdot}}(f)    , \quad f \in C^+(V).
\]
By dividing by $f$ (all allowed $f \in C(V)$ are strictly positive), this is equivalent to
\[
\Gamma_2^{\sqrt{\cdot}}(f)			   \geq \frac 1 {4d}   \left(\Delta^{\log} f \right)^2 +   K\Gamma^{\sqrt{\cdot}}(f)	,	\quad  f\in C^+(V).
\] 
By definition, this is equivalent to $CD^{\log }_{\sqrt{\cdot}}(4d,K)$.
This finishes the proof.
\end{proof}

%--------------------------------------------------------------------------------------------------------------------------------------------------------------
\section{The $CDE'$ inequality implies the $CD$ inequality} \label{SSCC}

First, we recall a limit theorem \cite[Theorem 3.18]{Münch2014} by which it is shown that the $CD\psi$ condition implies the $CD$ condition (cf. \cite[Corollary 3.20]{Münch2014}).
\begin{theorem} [Limit of the $\psi$-operators] 
Let $G=(V,E)$ be a finite graph. Then for all $f \in C(V)$, one has the pointwise limits 
\begin{eqnarray}
\lim_{\varepsilon \to 0} \frac 1 {\varepsilon} \Delta^\psi (1+ \varepsilon f)     \; =& \psi'(1) \Delta f 			   &\quad \mbox{for } \psi \in C^1(\posR), \label{ELL}\\
\lim_{\varepsilon \to 0} \frac 1 {\varepsilon^2} \Gamma^\psi (1+ \varepsilon f)   \; =& -\psi''(1) \Gamma(f) 	   	 &\quad \mbox{for } \psi \in C^2(\posR), \label{EL1}\\
\lim_{\varepsilon \to 0} \frac 1 {\varepsilon^2} \Gamma_2^\psi (1+ \varepsilon f) \; =& -\psi''(1) \Gamma_2(f) 		 &\quad \mbox{for } \psi \in C^2(\posR). \label{EL2}
\end{eqnarray}
Since all $f \in C(V)$ are bounded, one obviously has $1+ \varepsilon f> 0$ for small enough $\varepsilon>0$.
\end{theorem}
\begin{proof} For a proof, we refer the reader to the proof of \cite[Theorem 3.18]{Münch2014}.
\end{proof}
By adapting the methods of the proof of \cite[Corollary 3.20]{Münch2014}, we can show that $CD_\psi^\varphi$ implies $CD$ and, especially, we can handle the $CDE'$ condition.

\begin{theorem}
Let $\varphi, \psi \in C^2(\posR)$ be concave with $\psi''(1)\neq 0 \neq \varphi'(1)$ and let $d \in \posR$. Let $G=(V,E)$ be a graph satisfying the $CD_\psi^\varphi(d,K)$ condition.
Then, $G$ also satisfies the $CD\left(\frac{ -\psi''(1)}{\varphi'(1)^2}d,K \right)$ condition.
\end{theorem}
\begin{proof}
Let $f \in C(V)$.  We apply \cite[Theorem 3.18]{Münch2014} in the following both equations and since $G$ satisfies the $CD_\psi^\varphi(d,K)$ condition,
\begin{eqnarray*}
		      -\psi''(1)\Gamma_2(f)
=		      \lim_{\varepsilon \to 0} \frac 1 {\varepsilon^2} \Gamma_2^\psi(1+\varepsilon f) 
&\geq&    \lim_{\varepsilon \to 0}  \frac 1 {\varepsilon^2}\left( \frac 1 d \left[ \Delta^\varphi(1+\varepsilon f)\right]^2  + K  \Gamma^\psi(1+\varepsilon f) \right) \\
&=&			  \frac {\varphi'(1)^2 }{d} (\Delta f)^2 -\psi''(1)K \Gamma(f).
\end{eqnarray*}
Since $\psi$ is concave and $\psi''(1) \neq 0$, one has $-\psi''(1)>0$. Thus, we obtain that $G$ satisfies the $CD\left(\frac{ -\psi''(1)}{\varphi'(1)^2}d,0 \right)$ condition.
\end{proof}

\begin{corollary}
If $G=(V,E)$ satisfies the $CDE'(d,K)$, i.e., the $CD^{\log }_{\sqrt{\cdot}}(4d,K)$, then $G$ also satisfies the $CD(d,K)$ condition since $-4\sqrt{\cdot}''(1) = 1 = \log'(1)$.
\end{corollary}

%--------------------------------------------------------------------------------------------------------------------------------------------------------------
\section{The $CDE'$ inequality on Ricci-flat graphs} \label{SRF}

In \cite{Horn2014}, the $CDE'$ inequality is introduced. Examples for graphs satisfying this inequality have not been provided yet.
In this section, we show that the more general $CD_\psi^\varphi$ condition holds on Ricci-flat graphs (cf. \cite{Chung1996}). We will refer to the proof of the $CD\psi$ inequality on Ricci-flat graphs (cf. \cite[Theorem 6.6]{Münch2014}).
Similarly to \cite{Münch2014}, we introduce a constant $C_\psi^\varphi$ describing the relation between the degree of the graph and the dimension parameter in the $CD_\psi^\varphi$ inequality.

\begin{defn}
Let $\varphi,\psi \in C^1(\IR)$. 
Then for all $x,y > 0$, we write
\[
\widetilde{\psi}(x,y):= \left[\psi'(x) + \psi'(y) \right](1-xy) + x [\psi(y) - \psi(1/x)] + y [\psi(x) - \psi(1/y)]
\]
and
\[
C_\psi^\varphi := \inf_{(x,y) \in A_\varphi} \frac{\widetilde{\psi}(x,y)}{(\varphi(x) + \varphi(y) - 2\varphi(1))^2} \in [-\infty, \infty]
\]
with $A_\varphi:=\{(x,y)\in \left(\posR \right)^2 : \varphi(x) + \varphi(y) \neq 2\varphi(1)  \}$. We have $C_\psi^\varphi = \infty$ iff $A_\varphi = \emptyset$.
\end{defn}

\begin{theorem}[$CD_\psi^\varphi$ for Ricci-flat graphs] \label{TRicci} 
Let $D \in \IN$, let $G=(V,E)$ be a $D$-Ricci-flat graph, and let $\psi,\varphi \in C^1(\posR)$ be concave functions, such that $C_\psi^\varphi>0$. Then, $G$ satisfies the $CD_\psi^\varphi (d,0)$ inequality with $d = D / C_\psi^\varphi$.
\end{theorem}

\begin{proof}
We can assume $\psi(1)=0$ without loss of generality since $\Gamma_2^\psi$, $\Delta^\psi$ and $C_\psi$ are invariant under adding constants to $\psi$.
Let $v \in V$ and $f \in C(V)$. Since $G$ is Ricci-flat, there are maps $\eta_1,\ldots,\eta_D : N(v):=\{v\} \cup \{w \sim v\} \to V$ as demanded in Definition \ref{DefRF}. For all $i,j \in \{1,\ldots,D\}$, we denote
$%\begin{eqnarray*}
y 									:= f(v)$, $%\\
y_i  								:= f(\eta_i(v))$, $%\\
y_{ij} 							:= f(\eta_j(\eta_i(v)))$, $%\\
z_i 						 		:= y_i / y$, $% \\
z_{ij} 							:= y_{ij}/y_{i}.
$%\end{eqnarray*}

%From the proof of \cite[Theorem 6.6]{Münch2014}, we know that there is a permutation $i \mapsto i'$ such that
We take the sequence of inequalities at the end of the proof of \cite[Theorem 6.6]{Münch2014}.
First, we extract the inequality
\[
2 \Gamma_2^\psi (f) (v) \geq \frac 1 2  \sum_i \widetilde{\psi}(z_i, \zi).
\]
with $\psi$ and for an adequate permutation $i \mapsto i'$.

Secondly instead of continuing this estimate as in the proof of \cite[Theorem 6.6]{Münch2014}, we take the latter part applied with $\varphi$ instead of $\psi$ to see
\[
\frac 1 2  \sum_i  \left[ \varphi(z_i) + \varphi(\zi) \right]^2 \geq \frac { 2 }{D} \left[ \Delta^\varphi f (v) \right]^2.
\]
Since $\widetilde{\psi}(z_i, \zi) \geq   C_\psi^\varphi \left[ \varphi(z_i) + \varphi(\zi) \right]^2$, we conclude
\[
2 \Gamma_2^\psi (f) (v) \geq \frac { 2 C_\psi^\varphi }{D} \left[ \Delta^\varphi f (v) \right]^2.
\]
This finishes the proof.
\end{proof}

The above theorem reduces the problem, whether $CD_\psi^\varphi$ holds on Ricci-flat graphs, to the question whether $C_\psi^\varphi>0$.
By using this fact, we can give the example of the $CDE'$ condition on Ricci-flat graphs.
\begin{example}
Numerical computations indicate that $C^{\log }_{\sqrt{\cdot}} > 0.1104$.
Consequently by Theorem~\ref{TRicci}, $d$-Ricci-flat graphs satisfy the $CD^{\log }_{\sqrt{\cdot}}(9.058d,0)$ inequality and thus due to Proposition~\ref{PCC}, also the $CDE'(2.265d,0)$ inequality.
%An analytic proof of $C^{\log }_{\sqrt{\cdot}} > 0$ seems possible but lengthy, so we omit it here.
\end{example}
Now, we give an analytic estimate of $C^{\log }_{\sqrt{\cdot}}$ by using methods similar to the proof of \cite[Example 6.11]{Münch2014} which shows $C_{\log}^{\log} \geq 1/2$.

\begin{lemma}
$C^{\log }_{\sqrt{\cdot}} \geq 1/16 = 0.0625 $.
\end{lemma}
\begin{proof}
For $\psi=\sqrt{\cdot}$, we write
\begin{eqnarray*}
    \widetilde{\sqrt{\cdot}}(x,y) = \widetilde{\psi}(x,y)
&=& \left[\psi'(x) + \psi'(y) \right](1-xy) + x [\psi(y) - \psi(1/x)] + y [\psi(x) - \psi(1/y)] \\
&=& \left[\frac 1 {2 \sqrt{x}} + \frac 1 {2 \sqrt{y}} \right] (1-xy) + x \left[\sqrt{y} - \frac 1 {\sqrt{x}}  \right] +  y \left[\sqrt{x} - \frac 1 {\sqrt{y}}  \right] \\
&=& \frac{ \sqrt x + \sqrt y } 2 \cdot \left( \frac 1{\sqrt{xy}} - \sqrt{xy} \right) +  \left( \sqrt x + \sqrt y \right) \left( \sqrt{xy} - 1 \right) \\
&=& \frac{ \sqrt x + \sqrt y } 2 \cdot \left( (xy)^{1/4} - (xy)^{-1/4}  \right)^2 \\
&\geq&  (xy)^{1/4} \cdot \left( (xy)^{1/4} - (xy)^{-1/4}  \right)^2.
\end{eqnarray*}
Hence by substituting $e^{2t}:=(xy)^{1/4}$,
\begin{eqnarray*}
      \frac {\widetilde{\sqrt{\cdot}}(x,y)} {(\log x + \log y)^2}
&\geq&  (xy)^{1/4} \cdot \left( \frac {(xy)^{1/4} - (xy)^{-1/4}}    {4\log (xy)^{1/4}} \right)^2 \\
&=& e^{2t}  \cdot \left( \frac {e^{2t} - e^{-2t}}    {8t} \right)^2 \\
&=&  \left(  \frac {e^{3t} - e^{-t}}    {8t} \right)^2.
\end{eqnarray*}
We expand the fraction to
\[
\frac {e^{3t} - e^{-t}}    {8t} =  \frac {e^{3t} - e^{-t}}    {e^t-e^{-t}}  \cdot \frac {{e^t-e^{-t}}}    {8t} .
\]
Moreover,
\[
\frac {e^{3t} - e^{-t}}    {e^t-e^{-t}} = e^{2t} + 1 \geq 1
\]
and, by the estimate $\frac {\sinh t}{t} \geq 1$,
\[
\frac {{e^t-e^{-t}}}    {8t}  \geq 1/4.
\]
Putting together the above estimates yields
\[
C^{\log }_{\sqrt{\cdot}} = \inf_{x,y>0, xy\neq 1}  \frac {\widetilde{\sqrt{\cdot}}(x,y)} {(\log x + \log y)^2}
												 \geq (1/4)^2 = 1/16.
\]
This finishes the proof.
\end{proof}

\appendix
\section{Appendix}
\begin{defn}[Graph]
A pair $G=(V,E)$ with a finite set $V$ and a relation $E \subset V\times V$ is called a \emph{finite graph} if $(v,v) \notin E$ for all $v \in V$ and if $(v,w) \in E$ implies $(w,v) \in E$ for $v,w \in V$. 
For $v,w \in V$, we write $v \sim w$ if $(v,w) \in E$.% In this case, we say that the vertices $v$ and $w$ are \emph{adjacent}.
%For $v\in V$, we denote $\deg v := \#\{w \in V : w \sim v\}$. 
\end{defn}

\begin{defn}[Laplacian $\Delta$]
Let $G=(V,E)$ be a finite graph.
%The domain of the Laplacian $\Delta$ is
%$
%C(V):=  := \{f: V \to \IR\ \}.
%$ 
%$C(V):= \IR^V $
The \emph{Laplacian} $\Delta : C(V) := \IR^V \to C(V)$ is defined for $f \in C(V)$ and $v \in V$ as
$
\Delta f (v) := \sum_{w \sim v} (f(w) - f(v)).
$
\end{defn}

\begin{defn}
We write $\posR := (0,\infty)$ and $\nnegR := [0,\infty)$.
Let $G=(V,E)$ be a finite graph. Then, we write
$
C^+(V) := \{f:V \to \posR\}.
$
\end{defn}

\begin{defn}[Heat operator $\LL$]
Let $G=(V,E)$ be a graph.
%Let $\II \subset \\IR$ be an interval. 
%The domain of the heat operator $\LL$ is
%\[
%C^1(V \times \II) := \{u: V \times \II \to \IR \; | \; u \mbox { is continuously differentiable in the second variable} \}.
%\]
%For $u \in C^1(V \times \II)$ we write
%\[
%u_t(v):=u(v,t)
%\]
%for all $v \in V$ and $t \in \II$. 
%We call $t \in \II$ the \emph{time}, and we call $v \in V$ the \emph{location} respectively \emph{position}.
%The range of the heat operator is
%\[
%C(V \times \II) := \{u: V \times \II \to \IR \; | \;  u \mbox { is continuous in the second variable} \}.
%\]
The \emph{heat operator}
$\LL : C^1(V \times \II)  \to  C(V \times \II)$
 is defined by $\LL(u) := \Delta u  - \partial_t u$ for all $u \in C^1(V \times \II)$.
We call a function $\sol$ a \emph{solution to the heat equation} on $G$ if  $\LL(u) = 0$. 
\end{defn}

\begin{defn}[Ricci-flat graphs]  \label{DefRF}
Let $D \in \IN$. A finite graph $G = (V,E)$ is called $D$\emph{-Ricci-flat} in $v \in V$ if all $w \in N(v):=\{v\} \cup \{w \in V: w \sim v\}$ have the degree $D$, and if there are maps $\eta_1,\ldots,\eta_D : N(v) \to V $, such that for all $w \in N(v)$ and all $i, j \in \{1,\ldots,D\}$ with $i \neq j$, one has
$\eta_i(w) \sim w$,  
$\eta_i(w) \neq \eta_j(w)$, 
$\bigcup_k \eta_k(\eta_i(v)) = \bigcup_k \eta_i(\eta_k(v))$.
The graph $G$ is called $D$\emph{-Ricci-flat} if it is $D$-Ricci-flat in all $v \in V$.

\end{defn}

%-------------------------------------------------------------------------------------------------------------------
\subsection{The $CD$ condition via $\Gamma$ calculus}

We give the definition of the $\Gamma$-calculus and the $CD$ condition following \cite{Bakry1985}. 

\begin{defn}[$\Gamma$-calculus] \label{DG}
Let $G=(V,E)$ be a finite graph.
Then, the \emph{gradient form} or \emph{carré du champ} operator $\Gamma : C(V) \times C(V) \to C(V)$ is defined by
\[
2 \Gamma (f,g) := \Delta(fg) - f\Delta g - g\Delta f.
\]
Similarly, the \emph{second gradient form} $\Gamma_2 : C(V) \times C(V) \to C(V)$ is defined by
\[
2 \Gamma_2 (f,g) := \Delta \Gamma (f, g) - \Gamma(f, \Delta g)  - \Gamma (g, \Delta f).
\]
We write $\Gamma (f):= \Gamma (f,f)$ and $\Gamma_2 (f):= \Gamma_2 (f,f)$.
\end{defn}

\begin{defn}[$CD(d,K)$ condition] \label{DCD}
Let $G=(V,E)$ be a finite graph and $d \in \posR$.
We say $G$ satisfies the \emph{curvature-dimension inequality} $CD(d,K)$ if for all $f \in C(V)$, 
\[
\Gamma_2(f) \geq \frac 1 d (\Delta f)^2 + K \Gamma(f).
\]
We can interpret this as meaning that that the graph $G$ has a dimension (at most) $d$ and a Ricci curvature larger than $K$.
\end{defn}

%-------------------------------------------------------------------------------------------------------------------
\subsection{The $CDE$ and $CDE'$ conditions via $\widetilde{\Gamma_2}$}

We give the definitions of $CDE$ and $CDE'$ following \cite{Bauer2013,Horn2014}

\begin{defn}[The $CDE$ inequality] \label{DGT}
We say that a graph $G=(V,E)$ satisfies the $CDE(x,d,K)$ inequality if for any $f \in C^+(V)$ such that $\Delta f (x) <0$, we have
\[
\widetilde{\Gamma_2}(f) (x) := \Gamma_2(f)(x) - \Gamma \left(f, \frac{\Gamma(f)}{f} \right)(x) \geq \frac 1 d \left(\Delta f \right)^2(x) + K \Gamma(f)(x).
\]
We say that $CDE(d,k)$ is satisfied if $CDE(x,d,K)$ is satisfied for all $x \in V$.
\end{defn}

\begin{defn}[The $CDE'$ inequality]
We say that a graph $G=(V,E)$ satisfies the $CDE'(d,K)$ inequality if for any $f \in C^+(V)$, we have
\[
\widetilde{\Gamma_2}(f)  \geq \frac 1 d f^2 \left(\Delta \log f \right)^2 + K \Gamma(f).
\]
\end{defn}

%-------------------------------------------------------------------------------------------------------------------
\subsection{The $CD\psi$ conditions via $\Gamma^\psi$ calculus}

We give the definition of the $\Gamma^\psi$-calculus and the $CD\psi$ condition following \cite{Münch2014}.

\begin{defn}[$\psi$-Laplacian $\Delta^\psi$] \label{DLP}
Let $\psi \in C^1(\posR)$ and let $G=(V,E)$ be a finite graph. Then, we call
$\Delta^\psi : C^+(V) \to C(V)$, defined as 
\[
(\Delta^\psi f ) (v) :=  \left( \Delta \left[ \psi \left( \frac f {f(v)} \right) \right] \right) (v),
\]
the $\psi$\emph{-Laplacian}. 
\end{defn}

\begin{defn}[$\psi$-gradient $\Gamma^\psi$] \label{DGpsi}
 Let $\psi \in C^1(\posR)$ be a concave function and let $G=(V,E)$ be a finite graph.
We define  
\[
\overline{\psi}(x):= \psi'(1)\cdot(x-1)  - (\psi(x) - \psi(1)). %\geq 0.
\] 
Moreover, we define the $\psi$-\emph{gradient} as $\Gamma^\psi : C^+(V) \to C(V)$,
\[
\Gamma^\psi  := \Delta^{\overline{\psi}}.
\]
\end{defn}

\begin{defn}[Second $\psi$-gradient $\Gamma_2^\psi$] \label{DefG2} 
Let $\psi \in C^1(\posR)$, and let $G=(V,E)$ be a finite graph. Then, we define
$\Omega^\psi : C^+(V) \to C(V)$ by 
\[
(\Omega^\psi f ) (v) :=  \left( \Delta \left[ \psi' \left( \frac f {f(v)} \right)  \cdot \frac f {f(v)} \left[ \frac{\Delta f} {f} - \frac{(\Delta f)(v)} {f(v)} \right]  \right] \right) (v).
\]

Furthermore, we define the \emph{second} $\psi$-\emph{gradient} $\Gamma_2^\psi : C^+(V) \to C(V)$ by
\[
2 \Gamma_2^\psi (f) := \Omega^\psi f + \frac {\Delta f \Delta^\psi f} f - \frac {\Delta \left(f \Delta^\psi f\right)} f.
\]
\end{defn}

\begin{defn}[$CD\psi$ condition]
Let $G=(V,E)$ be a finite graph, $K \in \IR$ and $d \in \posR$. We say $G$ satisfies the $CD\psi(d,K)$ \emph{inequality} if for all $f \in C^+(V)$, one has 
\begin{equation*}
  \Gamma_2^\psi( f) \geq \frac 1 d \left(\Delta^\psi f \right)^2 + K\Gamma^\psi( f). \label{CDpsi}
\end{equation*}
\end{defn}


\begin{thebibliography}{9}


\bibitem {Bakry1985} \textsc{D. Bakry, M. Émery},\ \textit{Diffusions hypercontractives. (French) [Hypercontractive
diffusions]}, Séminaire de probabilité, XIX, 1983/84, 177-206, Lecture Notes in
Math., 1123, Springer, Berlin, 1985.


\bibitem {Bauer2013}\textsc{F. Bauer, P. Horn, Y. Lin, G. Lippner, D. Mangoubi, S.-T. Yau},\ \textit{Li-Yau inequality on graphs}, 
to appear in J. Differential Geom., Arxiv: 1306.2561v2 (2013).

\bibitem {Chung1996} \textsc{F. Chung, S.-T. Yau},\ \textit{Logarithmic Harnack inequalities}, Math. Res. Lett. \textbf{3} (1996), 

\bibitem {Horn2014}\textsc{P. Horn, Y. Lin, S. Liu, S.-T. Yau},\ 
\textit{Volume doubling, Poincaré inequality and Gaussian heat kernel estimate for nonnegative curvature graphs},  Arxiv: 1411.5087v2 (2014).

\bibitem{Münch2014}\textsc{F. Münch},\ \textit{Li-Yau inequality on finite graphs via non-linear curvature dimension conditions}, Arxiv: 1412.3340v1 (2014).



 
\end{thebibliography}
\end{document}